\documentclass{article}

\usepackage[english]{babel}
\usepackage{amsthm}
\usepackage[T1]{fontenc}
\usepackage[utf8]{inputenc}
\usepackage[english]{babel}
\usepackage{graphicx, color}
\usepackage{amssymb}
\usepackage{amsmath, xfrac}
\usepackage{latexsym}
\usepackage{caption, subcaption}
\usepackage{lineno}
\usepackage{comment}
\usepackage{mathtools}  
\usepackage{hyperref}
\usepackage{cite}

\newcounter{contSect} \numberwithin{contSect}{section}
 \numberwithin{contSub}{subsection}

\newtheorem{theorem}[contSect]{Theorem}
\newtheorem{corollary}[contSect]{Corollary}
\newtheorem{lemma}[contSect]{Lemma}

\newtheorem{observation}[contSect]{Observation}
\newtheorem{problem}[contSect]{Problem}

\usepackage[letterpaper,top=2cm,bottom=2cm,left=3cm,right=3cm,marginparwidth=1.75cm]{geometry}

\usepackage{amsmath}
\usepackage{graphicx}

\title{Pair crossing number, cutwidth, and good drawings on arbitrary point sets}
\author{Oriol Solé Pi}
\date{Facultad de Ciencias, Universidad Nacional Autónoma de México, Circuito Exterior,
Coyoacán, 04510, Ciudad de México, México.}

\begin{document}

\maketitle

\abstract{Determining whether there exists a graph such that its crossing number and pair crossing number are distinct is an important open problem in geometric graph theory. We show that $\textit{cr}(G)=O(\mathop{\mathrm{pcr}}(G)^{3/2})$ for every graph $G$, this improves the previous best bound by a logarithmic factor. Answering a question of Pach and Tóth, we prove that the bisection width (and, in fact, the cutwidth  as well) of a graph $G$ with degree sequence $d_1,d_2,\dots,d_n$ satisfies $\mathop{\mathrm{bw}}(G)=O\big(\sqrt{\mathop{\mathrm{pcr}}(G)+\sum_{k=1}^n d_k^2}\big)$. Then we show that there is a constant $C\geq 1$ such that the following holds: For any graph $G$ of order $n$ and any set $S$ of at least $n^C$ points in general position on the plane, $G$ admits a straight-line drawing which maps the vertices to points of $S$ and has no more than $O\left(\log n\cdot\left(\mathop{\mathrm{pcr}}(G)+\sum_{k=1}^n d_k^2\right)\right)$ crossings. Our proofs rely on a modified version of a separator theorem for string graphs by Lee, which might be of independent interest.}

\maketitle

\section{Introduction}\label{sec:intro}

Throughout this text we consider only simple undirected graphs. 

Given a graph $G=(V,E)$, a \textit{drawing} of $G$ is a representation such that the vertices correspond to distinct points on the plane and the edges are represented by simple continuous curves which go from one endpoint to the other but do not pass through any point representing a vertex. We further assume that no two curves are tangent or share an infinite number of points, and no three curves have a point in common. 

The \textit{crossing number} of $G$, $\mathop{\mathrm{cr}}(G)$, is the least number of crossing points amongst all drawings of $G$. The crossing number in one of the most important and studied measures of non-planarity, and it is also relevant from a practical point of view, particularly in graph visualization and VLSI. Similarly, the \textit{pair crossing number}, $\mathop{\mathrm{pcr}}(G)$, corresponds to the minimum number of pairs of crossing edges in all drawings of $G$. The systematic study of the pair crossing number was started by Pach and Tóth \cite{whichcrossingnumber}. Deciding whether $\mathop{\mathrm{cr}}(G)=\mathop{\mathrm{pcr}}(G)$ for all graphs is a major open problem in geometric graph theory, and a considerable amount of effort has been put into bounding $\mathop{\mathrm{cr}}(G)$ in terms of $\mathop{\mathrm{pcr}}(G)$. Tóth \cite{tothinitial} showed that $\textit{cr}(G)=O(\mathop{\mathrm{pcr}}(G)^{7/4}\log^{3/2}\mathop{\mathrm{pcr}}(G))$ and, using the same technique, the bound was later improved  to $\textit{cr}(G)=O(\mathop{\mathrm{pcr}}(G)^{3/2}\log\mathop{\mathrm{pcr}}(G))$ \cite{nearoptimal}, \cite{slightlybetter}. We strengthen this result by getting rid of the logarithmic factor. 

\begin{theorem}\label{teo:pcrbound}
For every graph $G$ we have that $\textit{cr}(G)=O(\mathop{\mathrm{pcr}}(G)^{3/2})$.
\end{theorem}

One of the most successful graph parameters in the study of the crossing number has been the \textit{bisection width}, $\mathop{\mathrm{bw}}(G)$, which is the least number of edges whose removal disconnects the graph into two parts having no more than $\frac{2}{3} \lvert V \rvert$ vertices each. The main result connecting the bisection width and the crossing number (the precise statement varies slightly from text to text, see \cite{vlsilayouts},\cite{applicationscrossingnumber} for the version below) tells us that if $G$ is a graph with degree sequence $d_1,d_2,\dots,d_n$, then \[\mathop{\mathrm{bw}}(G)\leq1.58\sqrt{16\mathop{\mathrm{cr}}(G)+\sum_{k=1}^n d_k^2}.\]

From now on, we denote $\sum_{k=1}^n d_k^2$ by $\mathop{\mathrm{ssqd}}(G)$. Pach and Tóth \cite{thirteenproblems} asked whether there is a constant $C$ such that $\mathop{\mathrm{bw}}(G)\leq C\big(\sqrt{\mathop{\mathrm{pcr}}(G)}+\sqrt{\mathop{\mathrm{ssqd}}(G)}\big)$. Almost providing a positive answer to this question, Kolman and Matoušek \cite{matousekpcr} showed that $\mathop{\mathrm{bw}}(G)=O\big(\log n\sqrt{\mathop{\mathrm{pcr}}(G)+\mathop{\mathrm{ssqd}}(G)}\big)$.

The \textit{cutwidth} of $G$, $\mathop{\mathrm{cw}}(G)$, is the least integer $\ell$ for which there is an ordering $v_1,v_2,\dots,v_n$ of its vertices such that the number of edges which have one endpoint in $\{v_1,\dots,v_i\}$ and the other in $\{v_{i+1},\dots,v_n\}$ is at most $\ell$ for every $i$ with $1\leq i\leq n-1$. Djidjev and Vrt’o \cite{cutwidth} showed that $\mathop{\mathrm{cw}}(G)=O\big(\sqrt{\mathop{\mathrm{cr}}(G)+\mathop{\mathrm{ssqd}}(G)}\big)$. Since $\textit{cw}(G)\geq\mathop{\mathrm{bw}}(G)$, this result can be interpreted as a stronger version of the inequality correlating $\mathop{\mathrm{cr}}(G)$ to $\mathop{\mathrm{bw}}(G)$. Settling the problem mentioned in the previous paragraph, we show that this holds for the pair crossing number as well.

\begin{theorem}\label{teo:cutwidth}
For every graph $G$, $\mathop{\mathrm{cw}}(G)=O\big(\sqrt{\mathop{\mathrm{pcr}}(G)+\mathop{\mathrm{ssqd}}(G)}\big)$.
\end{theorem}

\begin{corollary}\label{teo:bwpcr}
For every graph $G$, $\mathop{\mathrm{bw}}(G)=O\big(\sqrt{\mathop{\mathrm{pcr}}(G)+\mathop{\mathrm{ssqd}}(G)}\big)$.
\end{corollary}

Since the pathwidth of $G$, $\mathop{\mathrm{pw}}(G)$, satisfies $\mathop{\mathrm{pw}}(G)\leq\mathop{\mathrm{cw}}(G)$, the above inequalities hold as well for this parameter. This had already been observed in \cite{cutwidth} with $\mathop{\mathrm{cr}}(G)$ in place of $\mathop{\mathrm{pcr}}(G)$. 

Corollary \ref{teo:bwpcr} can be used to extend several other results about the crossing number to the pair crossing number. Perhaps the most notable example of this is the following.

\begin{corollary}
Let $p\in(0,1)$ and $e=p\binom{n}{2}$. If $e>10n$ then, with high probability\footnote{An event which depends on $n$ is said to happen \textit{with high probability} (or w.h.p., for short) if the probability that it occurs goes to $1$ as $n$ goes to infinity.}, the pair crossing number of the random graph $G(n,p)$ is larger than $ce^2$ for some absolute constant $c$. 
\end{corollary}

Just as the analogous result for $\mathop{\mathrm{cr}}(G(n,p))$ \cite{thirteenproblems}, this follows from the fact that, under the assumptions of the corollary, $\mathop{\mathrm{bw}}(G(n,p))$ is linear on $e$ w.h.p.. This has the nice consequence that, w.h.p., $\mathop{\mathrm{cr}}(G(n,p))$ and $\mathop{\mathrm{pcr}}(G(n,p))$ differ by no more than a multiplicative constant (clearly, both $\mathop{\mathrm{cr}}(G(n,p))$ and $\mathop{\mathrm{pcr}}(G(n,p))$ are $O(e^2)$ w.h.p.). Some other particularly interesting results that can be extended using Corollary \ref{teo:bwpcr} can be found in \cite{monotoneproperties} and \cite{untangling} (in fact, we will need one of these results later on).

Using the graph drawing technique developed by Bhatt and Leighton \cite{bhattleighton} and Even et al. \cite{VLSI} in conjunction with a graph partitioning result obtained by applying the bound $\mathop{\mathrm{bw}}(G)=O\big(\log n\sqrt{\mathop{\mathrm{pcr}}(G)+\mathop{\mathrm{ssqd}}(G)}\big)$ bound to a suitable auxiliary graph, Kolman and Matoušek \cite{matousekpcr} showed that $\mathop{\mathrm{cr}}(G)=O(\log^3 n(\mathop{\mathrm{pcr}}(G)+\mathop{\mathrm{ssqd}}(G)))$. We improve this by a factor of $\log^2 n$; this seems to be best bound that can possibly be attained by means of the drawing framework in \cite{bhattleighton}, \cite{VLSI}. 

\begin{theorem}\label{teo:pcrregularbound}
For every graph $G$ on $n$ vertices, $\mathop{\mathrm{cr}}(G)=O(\log n(\mathop{\mathrm{pcr}}(G)+\mathop{\mathrm{ssqd}}(G)))$.
\end{theorem}

There is a wide variety of graphs for which the above bound is better than the one provided by Theorem \ref{teo:pcrbound} (e.g., graphs with bounded degree and crossing number at least $n$).

A \textit{convex drawing} of a graph is a drawing in which the vertices correspond to a collection of points in convex position and the edges are represented by segments. It is known (\cite{VLSI}, \cite{convexcr}) that any graph $G$ has a convex drawing with no more than $O(\log n(\mathop{\mathrm{cr}}(G)+\mathop{\mathrm{ssqd}}(G)))$ crossings. Shahrokhi et al. \cite{convexcr} also proved that this was asymptotically optimal. The proof of Theorem \ref{teo:pcrregularbound} shows that that this result can be strengthened by substituting $\mathop{\mathrm{cr}}(G)$ by $\mathop{\mathrm{pcr}}(G)$, that is, $G$ has a convex drawing with at most $O(\log n(\mathop{\mathrm{pcr}}(G)+\mathop{\mathrm{ssqd}}(G)))$ crossings. We extend this result in the following way. For a planar set of points in general position, $S$, and a graph $G=(V,E)$ with $\lvert V\rvert\leq\lvert S\rvert$, a \textit{drawing of $G$ on $S$} is a drawing in which the vertices are represented by distinct points of $S$ and the edges are represented by segments.

\begin{theorem}\label{teo:generaldrawing}
There is an absolute constant $C\geq1$ with the following property: For any graph $G$ on $n$ vertices and any set $S$ of at least $n^C$ points in general position on the plane, there is a drawing of $G$ on $S$ with no more than \[O\left(\log n\cdot\left(\mathop{\mathrm{pcr}}(G)+\mathop{\mathrm{ssqd}}(G)\right)\right)\] crossings.
\end{theorem}

We remark that the logarithmic term above cannot be removed, as it was shown in \cite{convexcr} that every convex drawing of an $m\times m$ grid has at least $\Omega(m^2\log m)$ crossings. The technique used in the proof of Theorem \ref{teo:generaldrawing} can be thought of as a two dimensional version of the one in \cite{bhattleighton} and \cite{VLSI}. Notice that the weaker result with $2^{\Theta(n)}$ in place of $n^C$ follows directly from the celebrated Erdős-Szekeres theorem \cite{erdosszekeres} and the existence of convex drawings with $O\left(\log n\cdot\left(\mathop{\mathrm{pcr}}(G)+\mathop{\mathrm{ssqd}}(G)\right)\right)$ crossings.

A \textit{string graph} is an intersection graph of a collection of curves in the plane. The proof of the $\mathop{\mathrm{cr}}(G)$ bound in \cite{slightlybetter} depends on a separator theorem for string graphs by Lee \cite{Lee}. Our proofs of Theorems \ref{teo:pcrbound} and \ref{teo:cutwidth} rely on a modified version of the result by Lee, which might be of independent interest (see Theorem \ref{teo:edgevertexseparator}).

In Section \ref{sec:separators}, we prove the modified version of the separator theorem of Lee, which we mentioned above, and then use it to prove a couple graph partitioning results and Theorem \ref{teo:cutwidth}. Section \ref{sec:pcrbounds} is devoted to proving Theorems \ref{teo:pcrbound} and \ref{teo:pcrregularbound}. Section \ref{sec:drawing} contains the proof of Theorem \ref{teo:generaldrawing}. Finally, a couple of minor results and some open problems are discussed in Section \ref{sec:final}.

All logarithms in this text are base $2$. For a graph $G=(V,E)$ and a subset $U$ of $V$, we denote by $G[U]$ the subgraph of $G$ induced by $U$. The letter $c$ will be used in various sections to denote different constants; while this is a slight abuse of notation, it should not be the cause of any confusion.

\section{Separators, graph partitions, and cutwidth}\label{sec:separators}
Consider a graph $G=(V,E)$. For $r\in(0,1)$, a set of vertices $S\subset V$ is an $r$-\textit{balanced separator} if there is a partition $V=V_1\cup V_2\cup S$ such that no edge runs between $V_1$ and $V_2$, and $\lvert V_1 \rvert,\lvert V_2 \rvert\leq r\lvert V \rvert$.

Recall that a string graphs is an intersection graph of a collection of curves in the plane. The results in this revolve around  around the following result by Lee \cite{Lee}, which was originally conjectured by Fox and Pach \cite{fox_pach_2010} and later shown to be true up to a logarithmic factor by Matoušek \cite{nearoptimal}.

\begin{theorem}\label{teo:separator}
Every string graph with $m$ edges has a $\frac{2}{3}$-balanced separator of size $O(\sqrt{m})$.
\end{theorem}

For $r\in(0,1)$, we say that a set of vertices $S\subset V$ is an $r$-\textit{edge-balanced separator} if there is a partition $V=V_1\cup V_2\cup S$ such that there is no edge between $V_1$ and $V_2$, and the the induced subgraphs $G[V_1]$ and $G[V_2]$ contain at most $r\lvert E \rvert$ edges each. Using Theorem \ref{teo:separator}, we prove the following.

\begin{theorem}\label{teo:edgevertexseparator}
Let $r\in(\frac{2}{3},1)$, then every string graph with $m$ edges contains an $r$-edge-balanced separator of size $O_r(\sqrt{m})$, where the hidden constant depends only on $r$. Furthermore, if $r>\frac{3}{4}$ then the string graph contains a set of at $O_r(\sqrt{m})$ vertices that is simultaneously a  $\frac{2}{3}$-balanced separator and an $r$-edge-balanced separator.
\end{theorem}

\begin{proof}
Let $G=(V,E)$ be as in the theorem. We restrict ourselves to the case in which $G$ has no isolated vertex. 

First, we show that $G$ contains an $r$-edge-balanced separator of size $O_r(\sqrt{m})$. Consider a collection of curves realizing $G$, we may assume that no three curves go through the same point.  Let $t$ be a fixed positive integer, to be specified later.  We construct an auxiliary string in order to apply Theorem \ref{teo:separator}. For each pair of adjacent curves, choose an arbitrary common point between them and add $t$ pairwise disjoint \textit{auxiliary curves} around it that cross both of the original curves but intersect no other curve. Let $H\supset G$ be the string graph that corresponds to this new family of curves, then $H$ has $\lvert V \rvert+tm$ vertices and $(t+1)m$ edges. There is clearly a $\frac{2}{3}$-balanced separator for $H$ of minimal size which contains none of the auxiliary curves and, by Theorem \ref{teo:separator}, it has no more than $c_t\sqrt{m}$ vertices for some $c_t$ that depends only on $t$. Let $S_H$ be such a separator and consider the partition $V_{H,1}\cup V_{H,2}\cup S_H$ induced by $S_H$; we have that $\lvert V_{H,1}\rvert, \lvert V_{H,2}\rvert\leq \frac{2}{3}(\lvert V\rvert+tm)\leq \frac{2}{3}(t+2)m$. Notice that $S_H$ is also a separator in $G$ and that the two induced subgraphs $G[V_{H,1}\cap V]$ and $G[V_{H,2}\cap V]$ have at most $\lvert V_{H,1}\rvert/t$ and $\lvert V_{H,2}\rvert/t$ edges, respectively (each edge in the induced subgraphs corresponds to $t$ auxiliary curves). Since both of these quantities are at most $\frac{2}{3}\frac{t+2}{t}m$, taking $t$ large enough yields the result.

For the second part, we follow the proof of Lemma 5 in \cite{VLSI}. Let $S\subset V$ be a $\frac{2}{3}$-balanced separator of minimal size and $V=V_1\cup V_2\cup S$ be the partition induced by $S$, so that $G[V_1]$ has at most as many edges as $G[V_2]$. If $G[V_2]$ contains no more than $r\lvert E\rvert$ edges, then $S$ is the desired separator. Otherwise, take a small $\epsilon>0$ (to be specified later) and, within $G[V_2]$, consider a $\left(\frac{2}{3}+\epsilon\right)$-edge-balanced separator $S'\subset V_2$ of minimal size and the partition $V_2=V_3\cup V_4\cup S'$ induced by it; suppose that $\lvert V_3\rvert\leq\lvert V_4\rvert$. If $\lvert V_4\rvert\geq \frac{1}{3}\lvert V\rvert$, then we are done by taking the separator $S\cup S'$ and the partition $V=(V_1\cup V_3)\cup(V_4)\cup(S\cup S')$. Suppose that $\lvert V_4\rvert< \frac{1}{3}\lvert V\rvert$ and let $S''$ be a $\frac{2}{3}$-balanced separator of minimal size for $G[V_1]$ and $V_1=V_5\cup V_6\cup S''$ be the partition induced by it; assume that $\lvert V_5\rvert\leq\lvert V_6\rvert$. We claim that (if $\epsilon$ is small enough) the separator $Z=S\cup S'\cup S''$ and the partition $V=(V_3\cup V_6)\cup(V_4\cup V_5)\cup Z$ have the required properties. 

For starters, both $G[V_3\cup S\cup S'']$ and $G[V_4\cup S\cup S'']$ contain at least  $r\lvert E\rvert\left(\frac{1}{3}-\epsilon\right)$ edges. Since $r>\frac{3}{4}$, for small enough $\epsilon$ this quantity will be larger than $(1-r)\lvert E\rvert$. It follows that both $G[V_3\cup V_6]$ and $G[V_4\cup V_5]$ have at most $r\lvert E\rvert$ edges, so $Z$ is an $r$-edge-balanced separator.

Now we prove that $Z$ is a $\frac{2}{3}$-balanced separator. We have that \[\lvert V_3\rvert+\lvert V_6\rvert\leq\frac{1}{2}\lvert V_2\vert+\frac{2}{3}\lvert V_1\rvert\leq\frac{2}{3}\lvert V\rvert\] and \[\lvert V_4\rvert+\lvert V_5\rvert\leq\frac{1}{3}\lvert V\rvert+\frac{1}{2}\lvert V_1\rvert\leq\frac{1}{3}\cdot\frac{2}{3}\lvert V\rvert\leq\frac{2}{3}\lvert V\rvert.\] The result follows.
\end{proof}

Although we will require only the first part of the statement of Theorem \ref{teo:edgevertexseparator}, we believe the second one might be interesting on its own.

For any $G$ with degree sequence $d_1,d_2,\dots,d_n$, we denote $\sum_{i=1}^n \binom{d_i}{2}$ by $\mathop{\mathrm{bds}}(G)$. The graph partitioning result below will be the main tool in the proof of Theorem \ref{teo:cutwidth}.

\begin{lemma}\label{teo:partitionpcr}
For any graph $G$, there is a set consisting of $O\big(\sqrt{\mathop{\mathrm{pcr}}(G)+\mathop{\mathrm{bds}}(G)}\big)$ edges whose removal disconnects the graph into two (not necessarily connected) subgraphs $G_1$ and $G_2$ so that \[\mathop{\mathrm{pcr}}(G_1)+\mathop{\mathrm{bds}}(G_1)\leq \frac{3}{4}\left(\mathop{\mathrm{pcr}}(G)+\mathop{\mathrm{bds}}(G)\right),\]
\[\mathop{\mathrm{pcr}}(G_2)+\mathop{\mathrm{bds}}(G_2)\leq \frac{3}{4}\left(\mathop{\mathrm{pcr}}(G)+\mathop{\mathrm{bds}}(G)\right).\]
\end{lemma}

\begin{proof}
Consider a drawing $D$ of $G$ which exhibits $\mathop{\mathrm{pcr}}(G)$. For each pair of edges which cross and share an endpoint we draw a small auxiliary curve that intersects only one of them and no other edge. Consider the string graph $H$ induced by the closures of the resulting collections of curves (this way, any two edges sharing an endpoint correspond to adjacent vertices of $H$). Notice that each edge in this string graph corresponds, in a way, to either a pair of crossing edges of $G$ or a pair of edges of $G$ with a common endpoint. By Theorem \ref{teo:edgevertexseparator}, $H$ admits a $\frac{3}{4}$-edge-balanced separator of size $O\big(\sqrt{\mathop{\mathrm{pcr}}(G)+\mathop{\mathrm{bds}}(G)}\big)$. It is not hard to see that, after removing from it all of the auxiliary curves, this separator does the job.
\end{proof}

\begin{proof}[Proof of Theorem \ref{teo:cutwidth}]
Let $c$ be the hidden constant in Lemma \ref{teo:partitionpcr} and consider a set of at most $c\sqrt{\mathop{\mathrm{pcr}}(G)+\mathop{\mathrm{bds}}(G)}$ edges whose removal disconnects the graph into subgraphs $G_1$ and $G_2$, as in said corollary. By placing all vertices of $G_1$ before those of $G_2$ in the linear order and adding back in the deleted edges, we get that \[\mathop{\mathrm{cw}}(G)\leq\max\{\mathop{\mathrm{cw}}(G_1),\mathop{\mathrm{cw}}(G_2)\}+c\sqrt{\mathop{\mathrm{pcr}}(G)+\mathop{\mathrm{bds}}(G)},\] which, due to the bounds given by Lemma \ref{teo:partitionpcr}, solves to \[\mathop{\mathrm{cw}}(G)< c\sum_{i=0}^\infty\left(\sqrt{\frac{3}{4}}\right)^i\sqrt{\mathop{\mathrm{pcr}}(G)+\mathop{\mathrm{bds}}(G)}=O\left(\sqrt{\mathop{\mathrm{pcr}}(G)+\mathop{\mathrm{bds}}(G)}\right),\] as desired.
\end{proof}

The following artificial looking result will be important in Section \ref{sec:drawing}; we chose to include it here since it is closely tied to the proof of Theorem \ref{teo:cutwidth}.

\begin{lemma}\label{teo:carefulpartition}
Let $G$ be a simple graph with $n\geq 10000$ vertices. Consider a sequence of integers $0=a_0<a_1<a_2<\dots<a_{1000}=n$ such that $a_{i+1}-a_i\leq2(a_{j+1}-a_j)$ for all $0\leq i,j\leq 999$. Then there is a set consisting of $O\big(\sqrt{\mathop{\mathrm{pcr}}(G)+\mathop{\mathrm{bds}}(G)}\big)$ edges whose removal disconnects the graph into two (not necessarily connected) subgraphs $H_1$ and $H_2$ with $a_t$ and $n-a_t$ vertices ($1\leq t\leq 999$), respectively, so that \[\mathop{\mathrm{pcr}}(H_1)+\mathop{\mathrm{bds}}(H_1)\leq \frac{7}{8}\left(\mathop{\mathrm{pcr}}(G)+\mathop{\mathrm{bds}}(G)\right),\] \[\mathop{\mathrm{pcr}}(H_2)+\mathop{\mathrm{bds}}(H_2)\leq \frac{7}{8}\left(\mathop{\mathrm{pcr}}(G)+\mathop{\mathrm{bds}}(G)\right).\]
\end{lemma}

\begin{proof}
Notice that $a_{i+1}-a_i<20$ for all $1\leq i\leq 1000$. 

Consider a set of edges as in Lemma \ref{teo:partitionpcr} and let $H_1'$ and $H_2'$ be the two subgraphs that are obtained after the removal of said edges. Assume, w.lo.g., that the number of vertices of $H_1'$, which we denote by $r$, is at most $n/2$ and satisfies $a_j\leq r<a_{j+1}$. If $r=a_j$ then the set of edges has all the required properties. Otherwise, we will round the partition by moving $a_{j+1}-r$ vertices from $H_2'$ to $H_1'$. 

Apply Theorem \ref{teo:cutwidth} to $H_2'$ to get a linear ordering of its vertices. By cutting the $a_{j+1}-r$ vertices vertices on either extreme of this ordering and adjoining them to $H_1'$, we get a partition which still has no more than $O\big(\sqrt{\mathop{\mathrm{pcr}}(G)+\mathop{\mathrm{bds}}(G)}\big)$ edges running between the two parts; this could cause $\mathop{\mathrm{pcr}}(H_1)+\mathop{\mathrm{bds}}(H_1)$ to grow too large, however. In order to fix this issue, we observe that the proof of Theorem \ref{teo:cutwidth} actually provides us with a way of constructing many linear orderings which attain small cutwidth (indeed, at every step of the recursive process we get to choose which of the two subgraphs is placed on the left side and which one is placed on the right side). Since $H_2'$ contains at least $5000$ vertices and $a_{j+1}-r<a_{j+1}-a_i<20$, we can find several ($8$ is enough for our purposes) linear orderings which attain the desired cutwidth and such that no vertex of the graph appears amongst the first $a_{j+1}-r$ elements in more than one order. At least one of these pairwise disjoint sets of $a_{j+1}-r$ vertices will be such that moving its elements from $H_2'$ to $H_1'$ we get two graphs $H_1$ and $H_2$ that satisfy the required properties, or else we would be able to move all of the vertices in these sets to $H_1'$ simultaneously to obtain a subgraph $H$ of $G$ with $\mathop{\mathrm{pcr}}(H)+\mathop{\mathrm{bds}}(H)>\mathop{\mathrm{pcr}}(G)+\mathop{\mathrm{bds}}(G)$, which is clearly not possible.
\end{proof}

\section{Bounding the pair crossing number}\label{sec:pcrbounds}

In order to prove Theorem \ref{teo:pcrbound}, We combine Theorem \ref{teo:edgevertexseparator} with a redrawing technique by Tóth \cite{tothinitial}, \cite{slightlybetter}. We follow the terminology used in these papers. Given a drawing of $G$, we classify the edges as \textit{crossing edges} or \textit{empty edges} depending on whether they participate in crossing or not. Theorem \ref{teo:pcrbound} is an immediate consequence of the following lemma.

\begin{lemma}\label{teo:redraw}
Let $\mathcal{D}$ be a drawing of $G$ with $k\geq 0$ crossing pairs of edges and $l$ crossing edges. Then $G$ can be redrawn so that the empty edges remain the same, the crossing edges are redrawn in an arbitrarily small neighborhood of the original crossing edges, and the number of crossings is bounded by $clk^{1/2}$ for some absolute constant $c$.
\end{lemma}

\begin{proof}
Note that $l\leq 2k$. We prove the lemma by induction on $k$. Construct a string graph $H$ from the crossing edges of $\mathcal{D}$ (two vertices of $H$ are adjacent if and only if the corresponding edges cross each other). Let $S$ be a $\frac{3}{4}$-edge-balanced separator of size $O(\sqrt{k})$ for the graph $H$ and consider the vertex disjoint subgraphs $H_1$ and $H_2$ induced by it. Let $\mathcal{D}_1$ and $\mathcal{D}_2$ be the drawings formed by the edges of $H_1$ of and $H_2$, respectively, and denote by $k_1$ and $k_2$ the number of pairs of crossing edges in each of them. W.l.o.g., $k_1\geq k_2$. By the inductive hypothesis, these drawings can be modified so that they contain no more than $ck_1^{3/2}$ and $ck_2^{3/2}$ crossings, respectively, and the edges are drawn in a small neighborhood of the original edges. Color the edges in $\mathcal{D}_1$ and $\mathcal{D}_2$ blue and those in $S$ red, and let $BB$, $BR$ and $RR$ denote the number of crossings of each of the three possible kinds. As in \cite{tothinitial}, a simple procedure allows us to redraw the edges so that the triple $(BB,BR,RR)$ is not lexicographically larger than it was at the start, and any pair of edges crosses at most once. This results in a drawing with no more than \[clk_1^{1/2}+clk_2^{1/2}+O(\sqrt{k})\cdot l\leq cl(\frac{3}{4}k)^{1/2}+O(lk^{1/2})\] crossings. For any large enough $c$ the right hand side will be at most $clk^{1/2}$, and this concludes the proof. We refer the reader to \cite{tothinitial} or \cite{slightlybetter} for details about the redrawing step used above.
\end{proof}

\begin{proof}[Proof of Theorem \ref{teo:pcrbound}]
Consider a drawing of $G$ which attains $\mathop{\mathrm{pcr}}(G)$. By Lemma \ref{teo:redraw}, $G$ can be redrawn so that it has no more than $c\cdot 2\mathop{\mathrm{pcr}}(G)\cdot\mathop{\mathrm{pcr}}(G)^{1/2}=O(\mathop{\mathrm{pcr}}(G)^{3/2})$ crossings. The result follows.
\end{proof}

Next, we use the drawing framework of \cite{bhattleighton} and \cite{VLSI} to prove Theorem \ref{teo:pcrregularbound}. Again, we remark that Matoušek \cite{matousekpcr} obtained a slightly weaker result through the same technique. 

\begin{proof}[Proof of Theorem \ref{teo:pcrregularbound}]
We shall construct a drawing of $G$ with $O(\log n(\mathop{\mathrm{pcr}}(G)+\mathop{\mathrm{ssqd}}(G)))$ crossings. In order to achieve this, we build a linear ordering of $V$ as was done in the proof of Theorem \ref{teo:cutwidth} and then use it to embed the vertices on a circle arc in the natural way; the edges will simply be drawn as segments between the corresponding endpoints. We count the number of crossings in this drawing by looking at the partition tree that was already implicit in the proof of Theorem \ref{teo:cutwidth}, this step is a bit trickier than one might expect. We move on to the actual proof.

Lemma \ref{teo:partitionpcr} provides us with a partition of $G$ into subgraphs $G_1$ and $G_2$. We split the circle arc in half and recursively embed $G$ by placing the vertices of $G_1$ on one of the two smaller arcs and those of $G_2$ on the other one, and then adding back the edges of the separator. 

Let $T$ denote the rooted binary tree representing the recursive partitioning of $G$ used above. That is, the root corresponds to $V$ and every other vertex represents a proper subset of $V$, such that the two children of a non-leaf vertex $t\in T$ associated to a set $V_t$ correspond to the two sets constituting the partition of $V_t$ provided by Lemma \ref{teo:partitionpcr}. We denote the induced subgraph of $G$ with vertex set $V_t$ by $G_t$, notice that this is the graph that gets split at vertex $t$.

The endpoints of each edge of $G$ are separated at some step of the partitioning process; this gives a natural way of assigning to each edge $e$ of $G$ a vertex $t(e)$ of $T$. It is not hard to see that if two edges $e$ and $e'$ cross in our drawing of $G$, then either $t(e)$ is an ancestor of $t(e')$ or $t(e')$ or vice versa. We charge a crossing between $e$ and $e'$ to $e$ if $t(e)$ is an ancestor of $t(e')$, and we charge it to $e'$ otherwise. For each non-leaf vertex $t$ of $T$, denote by $E(t)$ the set of edges of $G$ that are assigned to $t$ (i.e., the set of edges which have an endpoint in each of the two sets represented by the children of $t$). For any two distinct vertices $u$ and $v$ of $G$, let $P(u,v)$ be the path in $T$ that connects the leafs representing $\{u\}$ and $\{v\}$. Even et al. \cite{VLSI} made the simple observation that for any edge $e=(u,v)$, summing $\lvert E(t)\rvert$ over all of the vertices in $P(u,v)$ yields an upper bound on the number of crossings that are charged to $e(t)$. Thus, the bounds provided by Theorem \ref{teo:partitionpcr} imply that the number of crossings that are charged to $e$ is at most \[O\left(\sum_{t\in P(u,v)}\sqrt{\mathop{\mathrm{pcr}}(G_t)+\mathop{\mathrm{bds}}(G_t)}\right)=O\left(\sqrt{\mathop{\mathrm{pcr}}(G_{t(e)})+\mathop{\mathrm{bds}}(G_{t(e)})}\right),\] where the equality follows from the exponential decay guaranteed by Theorem \ref{teo:partitionpcr}.

At the partition step corresponding to a vertex $t$ of $T$, the number of edges that were removed is $O\big(\sqrt{\mathop{\mathrm{pcr}}(G_t)+\mathop{\mathrm{bds}}(G_t)}\big)$. Whence, the total number of crossings charged to these edges is $O(\mathop{\mathrm{pcr}}(G_t)+\mathop{\mathrm{bds}}(G_t))$. Each level of $T$ corresponds to a partition of the vertex set of $G$. Thus, summing over all vertices at a fixed level of $T$, we get no more than $O(\mathop{\mathrm{pcr}}(G)+\mathop{\mathrm{bds}}(G))$ crossings (see Observation \ref{obs:partition} below). Since $T$ has depth $O(\log n)$, the total number of crossings in the drawing is at most \[O\left(\log n(\mathop{\mathrm{pcr}}(G)+\mathop{\mathrm{bds}}(G))\right),\] as desired.
\end{proof}

We implicitly used the observation below when counting the number of edges charged to each level of $T$. This observation was also used by Pach and Tardos in \cite{untangling}, and it will play an important role in the next section as well.

\begin{observation}\label{obs:partition}
If $V=V_1\sqcup V_2\sqcup\dots\sqcup V_k$ is a partition of the vertex set of $G$ and $G_i=G[V_i]$ is the subgraph induced by $V_i$ (for every $1\leq i\leq k$), then \[\sum_{i=1}^k\left( \mathop{\mathrm{pcr}}(G_i)+\mathop{\mathrm{bds}}(G_i)\right)\leq\mathop{\mathrm{pcr}}(G)+\mathop{\mathrm{bds}}(G)\] and, as a consequence, \[\sum_{i=1}^k\sqrt{ \mathop{\mathrm{pcr}}(G_i)+\mathop{\mathrm{bds}}(G_i)}\leq k^{1/2}\sqrt{\mathop{\mathrm{pcr}}(G)+\mathop{\mathrm{bds}}(G)}.\]
\end{observation}

\section{Good drawings on arbitrary point sets}\label{sec:drawing}

The partitioning result below is a consequence of Lemma \ref{teo:carefulpartition}; it is stated in a more general form than we will actually need.

\begin{lemma}\label{teo:partitiongraph}
Let $G$ be a graph on $n$ vertices and $s$ a positive integer.  Consider $2s$ non-negative integers $n_1,n_2,\dots,n_{2s}$ such that $\sum_{i=1}^{2s} n_i=n$ and $(n_{2i}+n_{2i+1})\leq2(n_{2j}+n_{2j+1})$ (for every $i$ and $j$ in $\{1,2,\dots,s\}$). Then, by deleting \[O\left(s^{1/2}\sqrt{\mathop{\mathrm{pcr}}(G)+\mathop{\mathrm{bds}}(G)}\right)\] edges, $G$ can be split into $2s$ (not necessarily connected) graphs $G_1,G_2,\dots,G_{2s}$ whose orders are $n_1,n_2,\dots, n_{2s}$ and such that \[\mathop{\mathrm{pcr}}(G_i)+\mathop{\mathrm{bds}}(G_i)\leq\left(\frac{1}{2}\right)^{\Theta(\log s)}(\mathop{\mathrm{pcr}}(G)+\mathop{\mathrm{bds}}(G))\] for $1\leq i\leq 2s$.
\end{lemma}

This lemma is close in spirit to Corollary 4 in \cite{untangling}, but it allows us to control the orders of the resulting subgraphs with much more precision.

\begin{proof}
We may and will assume that $s\geq 1000$. At a high level, the proof consists of repeatedly applying Lemma \ref{teo:carefulpartition} until every subgraph has the desired number of vertices. 

Let $m_i=n_{2i}+n_{2i+1}$ and partition the $m_i$’s into $1000$ blocks such that the ratio between the largest sum of a block and the smallest sum of a block is as little as possible; it is not hard to see that this ratio will not be larger than $2$. Denote by $S_1,S_2,\dots,S_{1000}$ the sums of the blocks and set $a_k=\sum_{i=1}^{k}S_i$. The $a_i$’s satisfy the conditions of Lemma \ref{teo:carefulpartition}, so we can choose $O\big(\sqrt{\mathop{\mathrm{pcr}}(G)+\mathop{\mathrm{bds}}(G)}\big)$ edges such that deleting them allows us to split $G$ into two subgraphs whose orders can be written as sums of two disjoint subcollections of the $m_i$’s; furthermore, each of these subcollections contains at least $1/2000$ of the $m_i$’s. We keep applying Lemma \ref{teo:carefulpartition} to partition the subgraphs that arise until the order of each of them corresponds to the sum of at most $1000$ of the $m_i$’s. The following argument, which will allows us to bound the number of edges that have been removed up to this point, is similar to the one used to prove Corollary 5 in \cite{untangling}. Consider the family $\mathcal{H}$ of all subgraphs of $G$ that appeared at some point during the decomposition procedure, we will assign a non negative integer to each element of $\mathcal{H}$, which we call the \textit{height} of the subgraph. The subgraphs that remain at the end of the process get assigned a $0$. All other heights are obtained recursively by following the following rule: if $H\in\mathcal{H}$ got split into subgraphs $H_1$ and $H_2$ (which are also in $\mathcal{H}$), then its height is one plus the maximum amongst the heights of $H_1$ and $H_2$. Each subgraph of height $j$ corresponds to a subcollection of at least $(1+1/2000)^j$ of the $m_i$’s, and it is not hard to see that any two subgraph with the same height are vertex disjoint. Thus, by Observation \ref{obs:partition}, we have \[\sum_{H \text{ has height } j}\sqrt{ \mathop{\mathrm{pcr}}(H)+\mathop{\mathrm{bds}}(H)}\leq \left(s\cdot\left(\frac{2000}{2001}\right)^j\right)^{1/2}\sqrt{\mathop{\mathrm{pcr}}(G)+\mathop{\mathrm{bds}}(G)}.\] Summing over all levels, we obtain that \[\sum_{H\in\mathcal{H}}\sqrt{ \mathop{\mathrm{pcr}}(H)+\mathop{\mathrm{bds}}(H)}<\sum_{j=0}^\infty\left(\frac{2000}{2001}\right)^{j/2}\cdot\left(s^{1/2}\sqrt{\mathop{\mathrm{pcr}}(G)+\mathop{\mathrm{bds}}(G)}\right).\] Whence the number of edges that have been removed up to this point is $O\big(s^{1/2}\sqrt{\mathop{\mathrm{pcr}}(G)+\mathop{\mathrm{bds}}(G)}\big)$. 

Next, for each subgraph $H\subset G$ that remains at the end of the process described above, we consider a linear ordering of its vertices which exhibits cutwidth $O\big(\sqrt{\mathop{\mathrm{pcr}}(H)+\mathop{\mathrm{bds}}(H)}\big)$ (recall that Theorem \ref{teo:cutwidth} guarantees the existence of such an order). Let $m_j,m_{j+1},\dots,m_{j+t}$ denote the $m_i$’s that correspond to $H$ ($t<1000$) and observe that, using the aforementioned ordering, $H$ can be partitioned into subgraphs $G_{2j},G_{2j+1},\dots,G_{2(j+t)+1}$ of orders $n_{2j},n_{2j+1},\dots,n_{2(j+t)+1}$ by deleting $O\big(\sqrt{\mathop{\mathrm{pcr}}(H)+\mathop{\mathrm{bds}}(H)}\big)$ edges (the number of cuts required is bounded by a constant). Again by Observation \ref{obs:partition}, the total number of edges that need to be removed in this last step is also $O\big(s^{1/2}\sqrt{\mathop{\mathrm{pcr}}(G)+\mathop{\mathrm{bds}}(G)}\big)$.

The desired bound on $\mathop{\mathrm{pcr}}(G_i)+\mathop{\mathrm{bds}}(G_i)$ follows directly from the bounds provided by Lemma \ref{teo:carefulpartition} and the fact that each $H\in\mathcal{H}$ was the result of applying said lemma $\Theta(\log s)$ times during the first part of the decomposition process.
\end{proof}

Consider a set $P=\{p_1,p_2,\dots,p_k\}$ of points in general position on the plane. For any $i_1<i_2<i_3$, we label the triple $p_{i_1},p_{i_2},p_{i_3}$ with either $+1$ or $-1$ depending on the orientation of the triangle $p_{i_1}p_{i_2}p_{i_3}$: this induces a mapping from $\binom{P}{3}$ to $\{-1,+1\}$, which we call the \textit{order type} of $P$. 

Given $k$ finite point sets $P_1,P_2,\dots,P_k$ on the plane, a \textit{transversal} of $(P_1,P_2,\dots,P_k)$ is a a $k$-tuple of points $(p_1,p_2,\dots,p_k)$ such that $p_i\in P_i$ for each $i$. We say that $(P_1,P_2,\dots,P_k)$ has \textit{same-type transversals} if all of its transversals have the same order type. We make the following simple observation.

\begin{observation}\label{obs:same-type}
If $(P_1,P_2,\dots,P_k)$ has same-type transversals and the convex hulls of the $P_i$s are pairwise disjoint, then every line intersects at most two of these convex hulls.
\end{observation}

Bárány and Valtr \cite{positivefraction} obtained a "same-type lemma" for point sets; the quantitative bounds given by this result were later improved by Fox et al. \cite{semialgebraic}, who showed the following. 

\begin{theorem}\label{teo:same-type}
Let $P_1,P_2,\dots,P_k$ ($k\geq 3$) be finite sets of points on the plane such that their union is in general position. Then one can find subsets $P_1'\subset P_1,P_2'\subset P_2,\dots,P_k'\subset P_k$ such that \[\lvert P_i'\rvert\geq2^{-O(k\log k)}\lvert P_i\rvert\] for all $i$ and  $(P_1',P_2',\dots,P_k')$ has same-type transversals.
\end{theorem}

\begin{corollary}\label{teo:same-type disjoint}
For every integer $k\geq 3$ there is a constant $c(k)>0$ such that the following holds. If $P$ is a finite set of points in general position on the plane and it contains at least $k$ elements, then there are $k$ subsets $P_1,P_2,\dots,P_k\subset P$ with pairwise disjoint convex hulls and \[\lvert P_1\rvert=\lvert P_2\rvert=\dots=\lvert P_k\rvert\geq c(k)\lvert P\rvert\] such that $(P_1,P_2,\dots,P_k)$ has same-type transversals.
\end{corollary}

\begin{proof}[Proof of Theorem \ref{teo:generaldrawing}]

Suppose that $\lvert S\rvert \geq n^C$ for some $C\geq 1$ to be chosen later. Let $k\geq3$ be an integer and assume that $n\geq k$. Apply Corollary \ref{teo:same-type disjoint} to $S$ and let $S_1,S_2,\dots,S_k\subset S$ be the resulting subsets. Now, we use Lemma \ref{teo:carefulpartition} to find $O\big(k^{1/2}\sqrt{\mathop{\mathrm{pcr}}(G)+\mathop{\mathrm{bds}}(G)}\big)$ edges whose deletion splits $G$ into $k$ graphs $G_1,G_2,\dots,G_k$ such that the orders of any two of them differ by at most $1$ and \[\mathop{\mathrm{pcr}}(G_i)+\mathop{\mathrm{bds}}(G_i)\leq\left(\frac{1}{2}\right)^{\Theta(\log k)}(\mathop{\mathrm{pcr}}(G)+\mathop{\mathrm{bds}}(G))\] (if $k$ is odd, we set $a_{2s}=0$). We fix a value of $k$ so that $\mathop{\mathrm{pcr}}(G_i)+\mathop{\mathrm{bds}}(G_i)\leq\frac{1}{5}(\mathop{\mathrm{pcr}}(G)+\mathop{\mathrm{bds}}(G))$ is guaranteed for every $i$; $k$ is independent of $G$ and $S$ and it will remain constant throughout the rest of the proof.

We shall construct a drawing which maps the vertices of $G_i$ to points in $S_i$ for every $i$; this is achieved by recursively applying the process above, narrowing the possible locations of the vertices at every step. To be precise, if at some point we have a subgraph $H\subset G$ of order at least $k$ whose vertices must be represented by points of a subset $S'\subset S$, then we use Corollary \ref{teo:same-type disjoint} to find $k$ subsets of $S'$ and Lemma \ref{teo:carefulpartition} to partition $H$ into $k$ subgraphs, which get assigned to the point sets. Repeat this until every resulting subgraph has order less than $k$ and then map the vertices (injectively) to any point in the corresponding subset of $S$. For the procedure to work, all the subsets of $S$ that remain at the end must contain at least $k-1$ points, this will be the case so long as \[\lvert S\rvert\geq k\cdot c(k)^{\lceil\log_k n\rceil},\] where $c(k)$ is the constant given by Corollary \ref{teo:same-type disjoint}. The above inequality holds if $C$ is large enough.

We bound the number of crossings in the resulting drawing as in the proof of Theorem \ref{teo:pcrregularbound}. The partition scheme of $G$ used above can be represented by a rooted tree $T$ (the leafs of $T$ correspond to vertices of $G$);  this time, $T$ has maximum degree $k$. Define $G_t$, $t(e)$ and $E(t)$ as before and for every vertex $t$ of $T$ let $\ell(t)$ denote the distance from $t$ to the root.

The fact that the sets produced by Corollary \ref{teo:same-type disjoint} have pairwise disjoint convex hulls implies that if two edges $e$ and $e'$ cross, then either $t(e)$ is an ancestor of $t(e')$ or $t(e')$ is an ancestor of $t(e)$. A was  done earlier, a crossing between edges $e$ and $e'$ is charged to either $e$ or $e'$, depending on which of $t(e)$ and $t(e')$ is closer to the root. Let $d\geq0$ be an integer. Consider an arbitrary edge $e$ of $G$ and write $t=t(e)$; by Observation \ref{obs:same-type}, there are at most $2^d$ vertices $t'$ of $T$ such that the following holds: $\ell(t')=\ell(t)+d$ and $e$ crosses at least one edge of $E(t')$. Recall that $\lvert E(t')\rvert=O\big(\sqrt{\mathop{\mathrm{pcr}}(G_t')+\mathop{\mathrm{bds}}(G_t')}\big)$, hence the choice of $k$ ensures that \[\lvert E(t')\rvert=O\left(\frac{1}{5^{d/2}}\sqrt{\mathop{\mathrm{pcr}}(G_t)+\mathop{\mathrm{bds}}(G_t)}\right).\] Summing over all possibilities for $t'$, we get at most \[O\left(\left(\frac{4}{5}\right)^{d/2}\sqrt{\mathop{\mathrm{pcr}}(G_t)+\mathop{\mathrm{bds}}(G_t)}\right),\] thus the number of crossings charged to $e$ is bounded by $O\big(\sqrt{\mathop{\mathrm{pcr}}(G_t)+\mathop{\mathrm{bds}}(G_t)}\big)$, and the total number of crossings charged to edges that were split at $t$ is at most $O\big(\lvert E(t)\rvert\sqrt{\mathop{\mathrm{pcr}}(G_t)+\mathop{\mathrm{bds}}(G_t)}\big)=O(\mathop{\mathrm{pcr}}(G_t)+\mathop{\mathrm{bds}}(G_t))$. As in the proof of Theorem \ref{teo:pcrregularbound}, Observation \ref{obs:partition} yields that the total number of crossings is no more than $O\left(\log n(\mathop{\mathrm{pcr}}(G)+\mathop{\mathrm{bds}}(G))\right)$.
\end{proof}

\section{Further research}\label{sec:final}

As mentioned in the introduction, determining whether $\mathop{\mathrm{cr}}(G)=\mathop{\mathrm{pcr}}(G)$ for every graph is still open. 

Lemma \ref{teo:redraw} performs particularly bad when the number of crossing edges and the number of crossings have the same order of magnitude, while Theorem \ref{teo:pcrregularbound} gives no information whatsoever for graphs whose crossing number is at most linear (in the number of vertices). By the celebrated crossing lemma (\cite{crossinglemma},\cite{crossinglemma2}), these situations can occur only for graphs with a linear number of edges. In an attempt to provide better bounds in the case that the number of edges is linear, we present a minor result concerning toroidal graphs\footnote{A graph is \textit{toroidal} if it can be drawn on a torus without crossings.}.

\begin{theorem}
Let $G$ be a toroidal graph of maximum degree $\Delta$, then  $\mathop{\mathrm{cr}}(G)=O(\Delta^2 \mathop{\mathrm{pcr}}(G))$.
\end{theorem}

\begin{proof}[Proof sketch]
Hliněný and Salazar \cite{toroidal} provided an $O(\Delta^2)$-approximation algorithm for the crossing number of toroidal graphs. The proof of the correctness of the approximation depends mainly on the two following facts: 

First, that $\mathop{\mathrm{cr}}(C_n\times C_n)=\Omega(n^2)$. While determining the exact value of  $\mathop{\mathrm{cr}}(C_n\times C_n)$ is an important open problem, this bound is well known (see \cite{cycleproducts}, for example). It is not hard check that this result holds for the pair crossing number as well. 

Secondly, if $H$ is a minor of $G$ and $H$ has maximum degree at most $4$ then $\mathop{\mathrm{cr}}(G)=\Omega(\mathop{\mathrm{cr}}(H))$ (\cite{minorcrfirst},\cite{minorcr}). The proof in \cite{minorcr} extends almost verbatim to show that, under these conditions, $\mathop{\mathrm{pcr}}(G)=\Omega(\mathop{\mathrm{pcr}}(H))$. These observations imply that the algorithm also provides an $O(\Delta^2)$-approximation for the pair crossing number, and the theorem follows.
\end{proof}

It seems likely that a similar result holds for bounded genus graphs. Wood and Telle \cite{planarwidth} introduced planar decompositions of graphs and showed that they are closely related to the crossing number. Using these decompositions, the proved that graphs with bounded degree which exclude a  fixed minor have linear crossing number, it might be interesting to try and extend some of their results to the pair crossing number. In another paper, Hliněný and Salazar \cite{almostplanar} studied the crossing number of almost planar graphs\footnote{A graph is \textit{almost planar} if it can be made planar by deleting a single edge.}, their results can be adapted to obtain the following theorem.

\begin{theorem}
If $G$ is almost planar of maximum degree $\Delta$, then $\mathop{\mathrm{cr}}(G)\leq\Delta\mathop{\mathrm{pcr}}(G)$. Furthermore, if $G$ can be made planar by deleting $k$ edges and it has maximum degree $\Delta$, then $\mathop{\mathrm{cr}}(G)= k\Delta\mathop{\mathrm{pcr}}(G)+k^2$.
\end{theorem}

The \textit{monotone crossing number} of $G$ is the least number of crossings over all monotone drawings\footnote{A drawing of a graph is \textit{monotone} if every vertical line intersects each edge in at most one point.}, and we denote it by $\mathop{\mathrm{mcr}}(G)$. The \textit{monotone pair crossing number}, $\mathop{\mathrm{mpcr}}(G)$, is defined analogously. Valtr \cite{monotone} proved that $\mathop{\mathrm{mcr}}(G)\leq 4\mathop{\mathrm{mpcr}}(G)$, and this result has not yet been improved.

Regarding Theorem \ref{teo:generaldrawing}, we have the following problem:

\begin{problem}
Determine the least constant $C$ such that the following holds: For any graph $G$ of order $n$ and any set $S$ of points in general position on the plane with $\lvert S\rvert\geq n^C$, there is a drawing of $G$ on $S$ with no more than \[O\left(\mathop{\mathrm{polylog}} n\cdot\left(\mathop{\mathrm{pcr}}(G)+\mathop{\mathrm{ssqd}}(G)\right)\right)\] crossings.
\end{problem}

We can not rule out the possibility that $C=1$ works.

The \textit{odd crossing number} of $G$, $\mathop{\mathrm{ocr}}(G)$, is the least number of pairs of edges with an odd number of crossings over all drawings of the graph. As with the pair crossing number, Pach and Tóth \cite{whichcrossingnumber} asked whether $\mathop{\mathrm{ocr}}(G)=\mathop{\mathrm{cr}}(G)$ for every graph, and showed that $\mathop{\mathrm{cr}}(G)=\mathop{\mathrm{ocr}}(G)^2$ always holds. Pelsmajer, Schaefer and  Štefankovič \cite{oddcrisnotcr} constructed a graph with $\mathop{\mathrm{ocr}}(G)\neq\mathop{\mathrm{cr}}(G)$; this was later improved by Tóth \cite{tothodd}, who presented examples of graphs satisfying $0.855\mathop{\mathrm{pcr}}(G)\geq\mathop{\mathrm{ocr}}(G)$. The $\mathop{\mathrm{cr}}(G)=\mathop{\mathrm{ocr}}(G)^2$ bound, however, has not been yet been improved. Pach and Tóth \cite{thirteenproblems} asked whether a result in the vein of Theorem \ref{teo:bwpcr} holds for the odd crossing number. Adapting a technique of Kolman and Matoušek \cite{matousekpcr}, Pach and Tóth \cite{bwocr} obtained that $\mathop{\mathrm{bw}}(G)=O\big(\log n^3\sqrt{\mathop{\mathrm{ocr}}(G)+\mathop{\mathrm{ssqd}}(G)}\big)$, but it seem like getting rid of the polylog factor would require some new ideas.

\bibliographystyle{plain}
\bibliography{refs}

\end{document}